\documentclass[12pt]{article}

\usepackage{amsmath, amsthm, amssymb, latexsym, enumerate}

\usepackage[all]{xy}

\newtheorem{thm}{Theorem}[section]
\newtheorem{lem}[thm]{Lemma}
\newtheorem{cor}[thm]{Corollary}

\theoremstyle{definition}

\newtheorem{defn}[thm]{Definition}
\newtheorem{q}[thm]{Question}

\renewcommand{\leq}{\leqslant}
\renewcommand{\geq}{\geqslant}

\DeclareMathOperator{\sg}{\textit{sg}}

\DeclareMathOperator{\vsg}{\textit{vsg}}

\DeclareMathOperator{\fs}{\textit{fs}}

\title{Thin Set Versions of Hindman's Theorem}

\renewcommand{\footnotemark}{}

\author{Denis R. Hirschfeldt and Sarah C. Reitzes\\Department of
Mathematics, University of Chicago\thanks{The authors were partially
supported by Focused Research Group grant DMS-1854279 from the
National Science Foundation of the United States. Hirschfeldt was also
partially support by NSF grant DMS-1600543, and Reitzes by NSF grant
DGE-1746045.}}

\begin{document}

\maketitle

\section{Introduction}

This paper is part of a line of research on the
computability-theoretic and reverse-mathematical strength of versions
of Hindman's Theorem~\cite{Hindman} that began with the work of Blass,
Hirst, and Simpson~\cite{BHS}, and has seen considerable interest
recently. We assume basic familiarity with computability theory and
reverse mathematics, at the level of the background material in
\cite{Hbook}, for instance. On the reverse mathematics side, the two
major systems with which we will be concerned are RCA$_0$, the usual
weak base system for reverse mathematics, which corresponds roughly to
computable mathematics; and ACA$_0$, which corresponds roughly to
arithmetic mathematics. For principles $P$ of the form $(\forall
X)\, [\Phi(X) \rightarrow (\exists Y)\, \Psi(X,Y)]$, we call any $X$
such that $\Phi(X)$ holds an \emph{instance} of $P$, and any $Y$ such
that $\Psi(X,Y)$ holds a \emph{solution} to $X$.

We begin by introducing some related combinatorial principles.  For a
set $S$, let $[S]^n$ be the set of $n$-element subsets of
$S$. \emph{Ramsey's Theorem} (RT) is the statement that for every $n$
and every coloring of $[\mathbb N]^n$ with finitely many colors, there
is an infinite set $H$ that is \emph{homogeneous} for $c$, which means
that all elements of $[H]^n$ have the same color. There has been a
great deal of work on computability-theoretic and reverse-mathematical
aspects of versions of Ramsey's Theorem, such as RT$^n_k$, which is RT
restricted to colorings of $[\mathbb N]^n$ with $k$ many colors. (See
e.g.\ \cite{Hbook}.)

The Thin Set Theorem is another variant of Ramsey's Theorem that has
been studied from this perspective. It follows easily from Ramsey's
Theorem itself.

\begin{defn}
\emph{Thin Set Theorem} (TS): For every $n$ and every coloring $c :
[\mathbb  N]^n \rightarrow \mathbb N$, there is an infinite set $T
\subseteq \mathbb N$ and an $i$ such that $c(s) \neq i$ for all $s \in
[T]^n$. We call such a set $T$ a \emph{thin set} for $c$. TS$^n$ is
the restriction of TS to colorings of $[\mathbb N]^n$.
\end{defn}

Jockusch~\cite{J} showed that there is a computable instance of
RT$^3_2$ such that any solution computes the halting problem
$\emptyset'$. As shown by Simpson~\cite{Sbook}, Jockusch's
construction can also be used to prove that RT$^3_2$ (and hence RT)
implies ACA$_0$ over RCA$_0$. Wang~\cite{W} showed that TS, on the
other hand, does not have this much power. Indeed, it has a property
known as strong cone avoidance, which implies in particular that for
every coloring $c : [\mathbb N]^n \rightarrow \mathbb N$ and every
noncomputable $X$, there is an infinite thin set for $c$ that does not
compute $X$. It also follows from strong cone avoidance that TS does
not imply ACA$_0$ over RCA$_0$.

As shown by Seetapun~\cite{SS}, RT$^2_k$ also fails to imply
ACA$_0$. Indeed, Liu~\cite{L1,L2} showed that it does not imply the
weaker system WKL$_0$, which consists of RCA$_0$ together with Weak
K\"onig's Lemma, or the even weaker system WWKL$_0$ consisting of
RCA$_0$ together with Weak Weak K\"onig's Lemma. Patey~\cite{P2}
showed that the same is true of TS.

We now turn to Hindman's Theorem. For a set $S \subseteq \mathbb N$,
let $\fs(S)$ be the set of sums of nonempty finite sets of distinct
elements of $S$.

\begin{defn}
\emph{Hindman's Theorem} (HT): For every coloring of $\mathbb N$ with
finitely many colors, there is an infinite set $S \subseteq \mathbb N$
such that all elements of $\fs(S)$ have the same color.
\end{defn}

Blass, Hirst, and Simpson~\cite{BHS} showed that such an $S$ can
always be computed in the $(\omega+1)$st jump of the coloring, and
that there is a computable coloring such that every such $S$
computes $\emptyset'$. By analyzing these proofs they showed that HT
is provable in ACA$_0^+$ (the system consisting of RCA$_0$ together
with the statement that $\omega$th jumps exist) and implies ACA$_0$
over RCA$_0$. The exact computability-theoretic and
reverse-mathematical strength of HT remains open.

There has recently been interest in studying restricted versions of HT
such as the following. (See e.g.\ \cite{Car}.)

\begin{defn}
HT$^{\leq n}$ is HT restricted to sums of at most $n$ many elements,
and HT$^{=n}$ is HT restricted to sums of exactly $n$ many
elements. HT$^{\leq n}_k$ and HT$^{=n}_k$ are the corresponding
restrictions to colorings with $k$ many colors.
\end{defn}

Dzhafarov, Jockusch, Solomon, and Westrick~\cite{DJSW} showed that
HT$^{\leq 3}_3$ implies ACA$_0$ over RCA$_0$. Carlucci,
Ko\l{}odzieczyk, Lepore, and Zdanowski~\cite{CKLZ} did the same for
HT$^{\leq 2}_4$. These principles are also complex in a more heuristic
sense: There is no known way to prove even HT$^{\leq 2}_2$ other than
to give a proof of the full HT, which has led Hindman, Leader, and
Strauss~\cite{HLS} to ask whether every proof of HT$^{\leq 2}$ is also
a proof of HT. This question can be formalized by asking whether
HT$^{\leq 2}$ (or HT$^{\leq 2}_2$) implies HT, say over RCA$_0$. A
related open question is whether HT$^{\leq 2}_2$ is provable in
ACA$_0$.

The principle HT$^{=2}$ is quite different, as HT$^{=2}_k$ follows
easily from RT$^2_k$. Indeed, it was not clear even whether this
principle is computably true until the work of Csima, Dzhafarov,
Hirschfeldt, Jockusch, Solomon, and Westrick~\cite{CDHJSW}, who showed
that it is not, and that indeed there is a computable instance of
HT$^{=2}_2$ with no $\Sigma^0_2$ solutions. (The same had been shown
for RT$^2_2$ by Jockusch~\cite{J}, who also showed that every
computable instance of RT$^2_2$ has a $\Pi^0_2$ solution, which
implies that the same is true of HT$^{=2}_2$.) They also showed that
there is  a computable instance of HT$^{=2}_2$ such that every
solution has DNC degree relative to $\emptyset'$, and adapted this
proof to show that HT$^{=2}_2$ implies the principle RRT$^2_2$, a
version of the Rainbow Ramsey Theorem, over RCA$_0$. (See Section
\ref{hard} for definitions.)

In this paper, we study further versions of Hindman's Theorem,
obtained by combining HT and its variants with the Thin Set Theorem.

\begin{defn}
thin-HT: For every coloring $c : \mathbb N \rightarrow \mathbb N$,
there is an infinite set $S \subseteq \mathbb N$ such that $\fs(S)$ is
thin for $c$. We definite restrictions such as thin-HT$^{\leq n}$
analogously.
\end{defn}

In Section~\ref{encoding}, we give similar lower bounds on the
complexity of thin-HT as Blass, Hirst, and Simpson~\cite{BHS} gave for
HT, which suggests that thin-HT behaves like HT at least to some
extent. Indeed, it seems possible that thin-HT is equivalent to HT
over RCA$_0$. The situation for restricted versions is different,
however. Clearly, thin-HT$^{=n}$ follows from TS$^n$, but in fact so
does thin-HT$^{\leq n}$, due to the following fact.

\begin{lem}
For each $n$ and $k$, the following holds in $\textup{RCA}_0 +
\textup{TS}^n$: Given $c_i : [\mathbb N]^{m_i} \rightarrow \mathbb N$
for $i \leq k$, with $m_i \leq n$ for all $i \leq k$, there is a
single infinite set $T$ and a $j$ such that $c_i(s) \neq j$ for each
$c_i$ and each $s \in [T]^{m_i}$ with $i \leq k$.
\end{lem}

\begin{proof}
We use the fact that TS$^n$ implies TS$^m$ for each $m<n$, and proceed
by external induction to prove the stronger assertion that for each $j
\leq k$, $\textup{RCA}_0 + \textup{TS}^n$ proves that there is an
infinite set $T$ and an infinite set $C$ such that $c_i(s) \notin C$
for each $c_i$ and each $s \in [T]^{m_i}$ with $i \leq j$.

We do the base and inductive cases simultaneously. For $j+1>0$, assume
that that the assertion holds for $j$ and let $T$ and $C$ be as
above. For $j+1=0$, let $T=C=\mathbb N$. Define $d : [T]^{m_{j+1}}
\rightarrow \mathbb N$ as follows. Partition $C$ into infinitely many
infinite sets $A_0,A_1,\dots$. Let $d(s) = 0$ if either $c_{j+1}(s)
\in A_0$ or $c_{j+1}(s) \notin C$, and for $i>0$, let $d(s)=i$ if
$c_{j+1}(s) \in A_i$. By TS$^{m_{j+1}}$, there is an infinite $U
\subseteq T$ that is thin for $d$. Let $i \notin d([U]^{m_{j+1}})$ and
let $D=A_i$. Then $U$ and $D$ are infinite sets such that $c_i(s)
\notin D$ for each $c_i$ and each $s \in [U]^{m_i}$ with $i \leq j+1$.
\end{proof}

This lemma allows us to get thin-HT$^{\leq n}$ from TS$^n$ by taking a
coloring $c: \mathbb N \rightarrow \mathbb N$ and considering the
colorings that map $\{a_0,\ldots,a_j\}$ to $c(a_0+\cdots+a_j)$ for
each $j<n$.

There are also differences that have nothing to do with computability
theory and reverse mathematics between thin-HT$^{\leqslant n}$ on the
one hand, and thin-HT and HT$^{\leqslant n}$ on the other. The former
remains true if we allow sums of non-distinct elements, but it is not
difficult to show that the latter two do not. Similarly, the former
remains true for colorings $S \rightarrow \mathbb N$, where $S
\subseteq \mathbb N$ is any infinite set, while the latter two again
do not.

Nevertheless, even thin-HT$^{=2}$ still has a significant level of
complexity. In Section~\ref{hard}, we show that all of the lower
bounds mentioned above obtained in \cite{CDHJSW} for HT$^{=2}$ still
hold for thin-HT$^{=2}$.

In Section~\ref{open} we mention some open questions arising from our
results, and briefly discuss version of HT obtained by combining it
with thin set theorems for colorings with finitely many colors.

\section{Encoding $\emptyset'$ into thin-HT}
\label{encoding}

In this section, we show how to build on the proof of Theorem 2.2 of
Blass, Hirst, and Simpson~\cite{BHS}, which shows that there is a
computable instance of HT such that every solution computes
$\emptyset'$, to show that the same is true of thin-HT. We then derive
a reverse-mathematical consequence of our proof.

\begin{thm}
There is a computable instance of \textup{thin-HT} such that every
solution computes $\emptyset'$.
\end{thm}

\begin{proof}
As in the proof of Theorem 2.2 of \cite{BHS}, we write each number
$x>0$ as $2^{n_0}+\cdots+2^{n_k}$ with $n_0<\cdots<n_k$, and define
$\lambda(x)=n_0$ and $\mu(x)=n_k$. A set $S$ \emph{has $2$-apartness}
if for every $x,y \in S$ with $x<y$, we have
$\mu(x)<\lambda(y)$. Lemma 4.1 of \cite{BHS} shows that from any
infinite $S$ we can compute an infinite set $T$ with $2$-apartness
such that $\fs(T) \subseteq \fs(S)$ (and hence if $\fs(S)$ is thin for
a coloring, so is $\fs(T)$).

Let $x=2^{n_0}+\cdots+2^{n_k}$ with $n_0<\cdots<n_k$. Say that
$(n_i,n_{i+1})$ is a \emph{short gap} in $x$ if there is an $m<n_i$
such that $m \notin \emptyset'[n_{i+1}]$ but $m \in \emptyset'$. Say
that $(n_i,n_{i+1})$ is a \emph{very short gap} in $x$ if there is an
$m<n_i$ such that $m \notin \emptyset'[n_{i+1}]$ but $m \in
\emptyset'[n_k]$. Let $\sg(x)$ and $\vsg(x)$ be the numbers of short
gaps and very short gaps in $x$, respectively. Note that $\sg$ is not
a computable function, but $\vsg$ is.

Fix a bijection between $\mathbb N$ and the set of pairs $(p,i)$ where
$p$ is prime and $1 \leq i < p$, and identify $\mathbb N$ with this
set via this bijection. Define the coloring $c$ by letting
$c(x)=(p,i)$ where $p$ is the least prime that does not divide
$\vsg(x)$ and $\vsg(x) = i \bmod p$. We say that $x$ \emph{has color
$(p,i)$} if $c(x)=(p,i)$, and we also say that $x$ \emph{has color}
$(p,0)$ or $(p,p)$ if it has color $(q,i)$ for some $q>p$, i.e., if
every prime less than or equal to $p$ divides $\vsg(x)$.

Let $Y$ be such that $\fs(Y)$ is an infinite thin set for $c$. We can
assume that $Y$ has $2$-apartness, by Lemma 4.1 of \cite{BHS}, as
mentioned above. This condition ensures that if $x,y \in \fs(Y)$ and
$\mu(x)<\lambda(y)$, and we express $x$ and $y$ as sums of sets $F$
and $G$ of distinct elements of $Y$, respectively, then $F$ and $G$
are disjoint, and hence $x+y \in \fs(Y)$.  Say that $S \subseteq
\fs(Y)$ is \emph{$\lambda$-bounded} if there is a bound on the values
of $\lambda(x)$ for $x \in S$ (which includes the case
$S=\emptyset$). Note that $\fs(Y)$ itself is not
$\lambda$-bounded. Note also that the union of finitely many
$\lambda$-bounded sets is $\lambda$-bounded. Say that a color $j$ is
\emph{almost absent} from $\fs(Y)$ if the set of $x \in \fs(Y)$ that
have color $j$ is $\lambda$-bounded. (This definition includes the
case $j=(p,0)$, or equivalently $j=(p,p)$.)

\begin{lem}
There are $p$ and $0 \leq i < p$ such that $(p,i+1)$ is almost absent
from $\fs(Y)$ but $(p,i)$ is not.
\end{lem}

\begin{proof}
Let $p$ be least such that there is a $j$ for which $(p,j)$ is almost
absent from $\fs(Y)$, which exists since $\fs(Y)$ is thin. If $p=2$
then $(p,j+1)$ cannot be almost absent, since every number has color
$(p,j)$ or $(p,j+1)$. Now suppose that $p>2$ and $q$ is the preceding
prime. Since $(q,0)$ is not almost absent from $\fs(Y)$ and every
number that has color $(q,0)$ has color $(p,j)$ for some $j$, there
is some $k$ such that $(p,k)$ is not almost absent. In either case,
since having color $(p,0)$ is the same as having color $(p,p)$, the
lemma follows.
\end{proof}

Fix $p$ and $i$ as in the above lemma.

\begin{lem}
Let $1 \leq j <p$. Then $S=\{x \in \fs(Y) : \sg(x) = j \bmod p\}$ is
$\lambda$-bounded.
\end{lem}

\begin{proof}
Suppose $S$ is not $\lambda$-bounded. Let $q_0<\cdots<q_{m-1}$ be the
primes less than $p$. Since there are only finitely many sequences
$(k_0,\ldots,k_{m-1})$ with $k_i < q_i$, there is such a sequence for
which $T=\{x \in S : (\forall \ell<m)\, \sg(x) = k_\ell \bmod
q_\ell\}$ is not $\lambda$-bounded.

Since $j \neq 0 \bmod p$, and hence $q_0 \cdots q_{m-1} j \neq 0 \bmod
p$, there is a multiple $n$ of $q_0 \cdots q_{m-1}$ such that $nj = 1
\bmod p$ (where $q_0 \cdots q_{m-1} = 1$ if $p=2$). Since $T$ is not
$\lambda$-bounded, there are $x_0<\cdots<x_{n-1} \in T$ such that each
$\lambda(x_{k+1})$ is sufficiently large relative to $\mu(x_k)$ to
ensure that $(\mu(x_k),\lambda(x_{k+1}))$ is not a short gap. Then the
short gaps in $x_0+\cdots+x_{n-1}$ are exactly the short gaps in
$x_0,\ldots,x_{n-1}$, so
$\sg(x_0+\cdots+x_{n-1})=\sg(x_0)+\cdots+\sg(x_{n-1})$. The latter is
equal to $nj \bmod p = 1 \bmod p$, since each $x_\ell$ is in $S$, and
is also equal to $nk_\ell \bmod q_\ell$ for each $\ell<m$, and hence
equal to $0 \bmod q_\ell$ for each $\ell<m$, since $n = 0 \bmod
q_\ell$.

Since $(p,i)$ is not almost absent from $\fs(Y)$, there is a $y \in
\fs(Y)$ that has color $(p,i)$ such that $\lambda(y) > \mu(x_{n-1})$,
and every number less than $\mu(x_{n-1})$ that is in $\emptyset'$ is
already in $\emptyset'[\lambda(y)]$. Note that $\vsg(y) = 0 \bmod
q_\ell$ for each $\ell<m$, as otherwise $c(y)$ would be of the form
$(q_\ell,k)$ for some $1 \leq k < q_\ell$. Now $\vsg(x_0 + \cdots +
x_{n-1} + y) = \vsg(y) + \sg(x_0 + \cdots + x_{n-1})$, which is equal
to $i+1 \bmod p$, and to $0 \bmod q_\ell$ for all $\ell<m$. So $x_0 +
\cdots + x_{n-1} + y$ has color $(p,i+1)$. As we can choose $x_0$ so
that $\lambda(x_0)$ is arbitrarily large, $(p,i+1)$ is not almost
absent from $\fs(Y)$, contradicting the choice of $i$.
\end{proof}

So by removing finitely many elements from $Y$ if needed, we can
assume that $p$ divides $\sg(x)$ for all $x \in \fs(Y)$. We can now
argue as in the proof of Claim 2 in the proof Theorem 2.2 of
\cite{BHS} to compute $\emptyset'$ from $Y$: Given $n$, find $x,y \in
Y$ such that $x<y$ and $n < \mu(x)$. The short gaps in $x+y$ are the
ones in $x$, the ones in $y$, and possibly $(\mu(x),\lambda(y))$. But
if the latter is a short gap, then $\sg(x+y)=\sg(x)+\sg(y)+1$, which
is impossible since $p$ divides all three numbers. Thus $n \in
\emptyset'$ if{}f $n \in \emptyset'[\lambda(y)]$.
\end{proof}

The above proof can be carried out in relativized form in RCA$_0$
except for two issues: One is that in RCA$_0$ we cannot show that the
union of finitely many $\lambda$-bounded sets is $\lambda$-bounded,
which in general requires the $\Pi^0_1$-bounding principle. Another is
that being almost absent is a $\Sigma^0_2$ condition, so we cannot
conclude in RCA$_0$ that there is a least $p$ such that there is a $j$
for which $(p,j)$ is almost absent from $\fs(Y)$. Since
$\Pi^0_1$-bounding follows from $\Sigma^0_2$-induction over RCA$_0$,
adding the latter to RCA$_0$ is sufficient to get around these issues,
so we have the following.

\begin{thm}
\textup{thin-HT} implies \textup{ACA}$_0$ over $\textup{RCA}_0 +
\textup{I}\Sigma^0_2$.
\end{thm}

We do not know whether the use of I$\Sigma^0_2$ in this theorem can be
removed.

\section{Hard Instances of thin-HT$^{=2}$}
\label{hard}

In this section, we show that all the lower bounds on the complexity
of HT$^{=2}_2$ obtained by Csima, Dzhafarov, Hirschfeldt, Jockusch,
Solomon, and Westrick~\cite{CDHJSW} still hold for thin-HT$^{=2}$. (Of
course, all upper bounds on the complexity of HT$^{=2}_2$
automatically hold for thin-HT$^{=2}$, as the latter follows easily
from the former.) As in that paper, we use the computable version of
the Lov{\'a}sz Local Lemma due to Rumyantsev and Shen~\cite{R,RS}. In
particular, we use the following consequence of Corollary 7.2 in
\cite{RS} given in \cite{CDHJSW}, with an addendum on uniformity as
noted at the end of Section 4 of \cite{CDHJSW}. This uniformity, which
in \cite{CDHJSW} is used only to obtain results on Weihrauch
reducibility, will be essential in all our results, as their proofs
will require applying Theorem \ref{rsthm} infinitely often.

\begin{thm}[essentially Rumyantsev and Shen~\cite{RS}] 
\label{rsthm}
For each $q \in (0,1)$ there is an $M$ such that the following
holds. Let $F_0,F_1,\ldots$ be a computable sequence of finite sets,
each of size at least $M$. Suppose that for each $m \geq M$ and $n$,
there are at most $2^{qm}$ many $j$ such that $|F_j|=m$ and $n \in
F_j$, and that there is a computable procedure $P$ for determining the
set of all such $j$ given $m$ and $n$. Then there is a computable $c :
\mathbb N \rightarrow 2$ such that for each $j$ the set $F_j$ is not
homogeneous for $c$. Furthermore, $c$ can be obtained uniformly
computably from $F_0,F_1,\ldots$ and $P$ (for a fixed $q$).
\end{thm}

We will also rely in this section on arguments in \cite{CDHJSW} when
they carry through in this case in an entirely analogous way.

We now introduce a notion of largeness that will be key to our
iterated applications of Theorem \ref{rsthm}. As in \cite{CDHJSW}, we
will be diagonalizing against $\Sigma^0_2$ sets, so this notion will
be defined in terms of sets that are c.e.\ relative to $\emptyset'$.
For a set $A$ and a number $s$, we write $s+A$ for the set $\{s+a : a
\in A\}$. We write $W_e$ for the $e$th enumeration operator. Given $e$
and $s$, for each $x \in W^{\emptyset'}_e[s]$, let $t_x$ be the least
$t$ such that $x \in W^{\emptyset'}_e[u]$ for all $u \in
[t,s]$. (I.e., $t_x$ measures how long $x$ has been in
$W^{\emptyset'}_e$.) Order the elements of $W^{\emptyset'}_e[s]$ by
letting $x \prec y$ if either $t_x<t_y$ or both $t_x=t_y$ and
$x<y$. Let $E_e^n[s]$ be the set consisting of the least $n$ many
elements of $W^{\emptyset'}_e[s]$ under this ordering, or
$E_e^n[s]=[0,n)$ if $W^{\emptyset'}_e[s]$ has fewer than $n$ many
elements. If there is an $s$ such that $E_e^n[t] = E_e^n[s]$ for all
$t>s$ then let $E_e^n = E_e^n[s]$.

\begin{defn}
For a binary function $f$, say that a set $D$ is \emph{$f$-large} if
for all $e$ and $k$ such that $E_e^{f(e,k)}$ is defined, we have $|D
\cap (s+E_e^{f(e,k)})| \geq k$ for all sufficiently large $s$.
\end{defn}

Note that $\mathbb N$ is $g$-large for the function $g(e,k)=k$, and
that $f$-largeness is preserved under finite difference. The following
lemma captures the key property of this notion of largeness.

\begin{lem}
\label{mainlem}
From a binary function $f$ and an $f$-large set $D$, we can uniformly
compute a binary function $\widehat{f}$ and a splitting $D = D^0
\sqcup D^1$ such that each $D^i$ is $\widehat{f}$-large.
\end{lem}

Before proving this lemma, let us derive some of its consequences,
beginning with computability-theoretic lower bounds on the complexity
of thin-HT$^{=2}$. A function $f$ is \emph{diagonally noncomputable}
(\emph{DNC}) relative to an oracle $X$ if $f(e) \neq \Phi^X_e(e)$ for
all $e$ such that $\Phi^X_e(e)$ is defined, where $\Phi_e$ is the
$e$th Turing functional. A degree is \emph{DNC} relative to $X$ if it
computes a function that is DNC relative to $X$. An infinite set $A$
is \emph{effectively immune} relative to $X$ if there is an
$X$-computable function $f$ such that if $W^X_e \subseteq A$ then
$|W^X_e| < f(e)$.

\begin{thm}[Jockusch~\cite{Jfpf}]
\label{Jthm}
A degree is DNC relative to $X$ if and only if it computes a set that
is effectively immune relative to $X$.
\end{thm}

The proof of the following theorem shows how to obtain a hard
computable instance of thin-HT$^{=2}$ from Lemma \ref{mainlem}.

\begin{thm}
\label{dncthm}
There is a computable instance of \textup{thin-HT}$^{=2}$ such that
any solution is effectively immune relative to $\emptyset'$, and hence
has DNC degree relative to $\emptyset'$.
\end{thm}

\begin{proof}
Let $D_0=\mathbb N$ and $f_0(e,k)=k$. Given $D_n$ and $f_n$, let
$\widehat{f_n}$ and $D_n^i$ be as in Lemma \ref{mainlem}, let
$f_{n+1}=\widehat{f_n}$, and let $D_{n+1}=D_n^1$. Note that the $D_n$
are uniformly computable. Let $c(x)$ be the largest $n \leq x$ such
that $x \in D_n$. Then $c$ is a computable coloring of $\mathbb N$.
If $c(x)=n$ and $x>n$ then $x \in D_n$ but $x \notin D_m$ for $m>n$,
so $x \in D_n^0$. Thus for each $n$, we have that the difference
between $c^{-1}(n)$ and $D^0_n$ is finite, and hence $c^{-1}(n)$ is
$f_n$-large.

Let $S$ be a solution to $c$ as an instance of thin-HT$^{=2}$, and let
$n$ be such that $c(x+y) \neq n$ for all distinct $x,y \in S$. For any
$e$, if $|W_e^{\emptyset'}| \geq f_n(e,1)$ then $E_e^{f_n(e,1)}
\subseteq W_e^{\emptyset'}$ is defined, and hence $c^{-1}(n) \cap
(s+E_e^{f_n(e,1)}) \neq \emptyset$ for all sufficiently large $s$. In
other words, if $s$ is sufficiently large then there is an $x \in
E_e^{f_n(e,1)}$ such that $c(x+s)=n$. It follows that $E_e^{f_n(e,1)}
\nsubseteq S$, and hence $W_e^{\emptyset'} \nsubseteq S$, since
$E_e^{f_n(e,1)} \subseteq W_e^{\emptyset'}$. Thus we conclude that if
$W_e^{\emptyset'} \subseteq S$ then $|W_e^{\emptyset'}| <
f_n(e,1)$. Since $f_n(e,1)$ is computable as a function of $e$, it
follows that $S$ is effectively immune relative to $\emptyset'$, and
hence has DNC degree relative to $\emptyset'$.
\end{proof}

No infinite $\Sigma^0_2$ set can be effectively immune relative to
$\emptyset'$, so we have the following.

\begin{cor}
There is a computable instance of \textup{thin-HT}$^{=2}$ with no
$\Sigma^0_2$ solution.
\end{cor}

It follows that thin-HT is not provable in WKL$_0$, since the latter
has $\omega$-models consisting entirely of $\Delta^0_2$ sets. It was
noted in \cite{CDHJSW} that HT$^{=2}_2$ does not imply WKL$_0$, and
hence neither does thin-HT$^{=2}$. Thus thin-HT$^{=2}$ and WKL$_0$ are
incomparable over RCA$_0$. In fact, as mentioned in the introduction,
Patey~\cite{P2} showed that TS does not imply WKL$_0$, or even
WWKL$_0$, and we can easily adapt the proof of Theorem \ref{dncthm} to
thin-HT$^{=n}$ for any $n>2$, so we have the following.

\begin{cor}
For each $n>1$, both \textup{thin-HT}$^{=n}$ and
\textup{thin-HT}$^{\leq n}$ are incomparable with \textup{(W)WKL}$_0$
over \textup{RCA}$_0$.
\end{cor}

Arguing as in the proof of Corollary 3.6 of \cite{CDHJSW}, we have the
following.

\begin{cor}
There is a computable instance of \textup{thin-HT}$^{=2}$ such that
all solutions are hyperimmune.
\end{cor}

The reverse-mathematical analog of the existence of degrees that are
DNC over the jump is the principle 2-DNC, defined e.g.\ in Section 4
of \cite{CDHJSW}. Miller [unpublished] showed that 2-DNC is
equivalent, both over RCA$_0$ and in the sense of Weihrauch
reducibility, to the following version of the Rainbow Ramsey
Theorem, which was shown by Patey \cite{P} to be strictly weaker than
TS$^2$.

\begin{defn}
RRT$^2_2$: Let $c : [\mathbb N]^2 \rightarrow \mathbb N$ be such that
$|c^{-1}(i)| \leq 2$ for all $i$. Then there is an infinite set $R$
such that $c$ is injective on $[R]^2$.
\end{defn}

As discussed in \cite{CDHJSW}, the proof of Theorem \ref{rsthm}
carries through in RCA$_0$, from which it will follow that so does the
proof of Lemma \ref{mainlem} that we will give below. Thus the proof
of Theorem \ref{dncthm} also carries through in RCA$_0$, except for
one issue: Having $|W_e^{\emptyset'}| \geq m$ does not necessarily
imply in RCA$_0$ that $E^m_e$ is defined. (The issue is that RCA$_0$
does not imply the $\Pi^0_1$-bounding principle.) However, we can get
around this problem exactly as in Section 4 of \cite{CDHJSW}, by using
the principle 2-EI defined there, thus obtaining the following.

\begin{thm}
\textup{thin-HT}$^{=2}$ implies \textup{RRT}$^2_2$ over \textup{RCA}$_0$. 
\end{thm}

We can also obtain a Weihrauch reduction from RRT$^2_2$ to a
version of thin-HT$^{=2}$ as in the final paragraph of Section 4 of
\cite{CDHJSW}, but we have to be a bit careful because in the proof of
Theorem \ref{dncthm}, the function witnessing that $S$ is effectively
immune relative to $\emptyset'$ is obtained uniformly not from $S$,
but from an $n$ such that $c(x+y) \neq n$ for all distinct $x,y \in
S$. Let strong thin-HT$^{=2}$ be the version of thin-HT$^{=2}$
where a solution to an instance $c$ consists of both a solution $S$ to
$c$ as an instance of thin-HT$^{=2}$ and an $n$ as above. Then we have
the following.

\begin{thm}
\textup{RRT}$^2_2$ is Weihrauch-reducible to \textup{strong
thin-HT}$^{=2}$.
\end{thm}

We do not know, however, whether this theorem remains true if we
replace strong thin-HT$^{=2}$ by thin-HT$^{=2}$.

None of the above results depend on the addition function in
particular, and can be adapted as in \cite{CDHJSW} to any function $f:
[\mathbb N]^2 \rightarrow \mathbb N$ that is \emph{addition-like},
which means that
\begin{enumerate}

\item $f$ is computable,

\item there is a computable function $g$ such that $f(\{x,y\})>n$ for
all $y>g(x,n)$, and

\item there is a $b$ such that for all $x \neq y$, there are at most
$b$ many $z$'s for which $f(\{x,z\})=f(\{x,y\})$.

\end{enumerate}

We finish this section by proving Lemma \ref{mainlem}.

\begin{proof}[Proof of Lemma \ref{mainlem}]
Let $f$ be a binary function and $D$ an $f$-large set. We will apply
Theorem \ref{rsthm} to obtain a computable $c : \mathbb N \rightarrow
2$. We then define $D^i = \{n \in D : c(n)=i\}$. The value of $q$ will
not matter here, so let us fix $q=\frac{1}{2}$. Let $M$ be as in
Theorem \ref{rsthm}.

Let $g$ be a computable injective binary function with computable
image such that $kg(e,k) \leq 2^{\frac{g(e,k)}{2}}$ and $g(e,k) \geq
M$ for all $e$ and $k$.

Say that $s$ is \emph{acceptable for $e,k$} if $|D \cap (s +
E^{f(e,kg(e,k))}_e[s])| \geq kg(e,k)$ and for every $t<s$ such that $(s
+ E^{f(e,kg(e,k))}_e[s]) \cap (t + E^{f(e,kg(e,k))}_e[t]) \neq
\emptyset$, we have $ E^{f(e,kg(e,k))}_e[s] = E^{f(e,kg(e,k))}_e[t]$. If
$s$ is acceptable for $e,k$ then let $F_{e,k,s,0}$ be the first
$g(e,k)$ many elements of $s + E^{f(e,kg(e,k))}[s]$, let $F_{e,k,s,1}$ be
the next $g(e,k)$ many elements of $s + E^{f(e,kg(e,k))}[s]$, and so on,
until $F_{e,k,s,k-1}$.

Let $\mathcal F$ consist of all $F_{e,k,s,j}$ for all $e,k$, all $s$
acceptable for $e,k$, and all $j<k$. Then we can arrange the elements
of $\mathcal F$ into a computable sequence of finite sets, each of
size at least $M$. Fix $x$ and $m$. If $m$ is not in the image of
$g$ then there are no elements of $\mathcal F$ of size $m$. Otherwise,
there is a unique pair $e,k$ such that $m=g(e,k)$, and all elements of
$\mathcal F$ of size $m$ that contain $x$ are of the form
$F_{e,k,s,j}$ for some $s \leq x$. We can computably determine all
such sets from $m$ and $x$, and the definition of acceptability means
that there are at most $kg(e,k) \leq 2^{\frac{m}{2}}$ many such sets.

Thus the hypotheses of Theorem \ref{rsthm} hold, and hence there is a
$c$, obtained uniformly computably from $f$ and $D$, such that none of
the sets in $\mathcal F$ are homogeneous for $c$. Let
$\widehat{f}(e,k)=f(e,kg(e,k))$ and let $D^i = \{n \in D :
c(n)=i\}$. Fix $e$ and $k$ such that $E_e^{\widehat{f}(e,k)}$ is
defined. If $s$ is sufficiently large then $s$ is acceptable for
$e,k$, and $F_{e,k,s,j} \subseteq s + E_e^{\widehat{f}(e,k)}$ for all
$j<k$. For each $j<k$ and $i<2$, there is at least one $x \in
F_{e,k,s,j}$ such that $c(x)=i$. Since the $F_{e,k,s,j}$ are disjoint,
$|D^i \cap (s + E_e^{\widehat{f}(e,k)})| \geq k$. Thus $D^i$ is
$\widehat{f}$-large.
\end{proof}

\section{Open Questions}
\label{open}

In this section, we collect a few open questions and possible
directions for further work arising from the above results.

\begin{q}
Does thin-HT imply ACA$_0$ over RCA$_0$ (i.e., without assuming
I$\Sigma^0_2$)?
\end{q}

Of course, one way to give a positive answer to this question would be
to show that thin-HT implies I$\Sigma^0_2$ over RCA$_0$. If that is
not the case, then it could be interesting to try to determine the
first-order part of thin-HT.

\begin{q}
Is thin-HT provable in ACA$_0$?
\end{q}

\begin{q}
Does thin-HT imply HT, say over RCA$_0$?
\end{q}

In the spirit of Hindman, Leader, and Strauss~\cite{HLS}, we can also
ask the less formal question of whether there is a proof of thin-HT
that is not already a proof of HT.

\begin{q}
Is RRT$^2_2$ Weihrauch-reducible to thin-HT$^{=2}$ (as opposed to
strong thin-HT$^{=2}$)?
\end{q}

\begin{q}
What is the exact relationship between thin-HT$^{=2}$ and each of
TS$^2$, RRT$^2_2$, and HT$^{=2}$?
\end{q}

There are also versions of the Thin Set Theorem for colorings with
finitely many colors. For example, an instance of TS$^n_k$ is a
coloring $c$ of $[\mathbb N]^n$ with $k$ many colors, and a solution
to this instance is an infinite set $T$ such that $|c([T]^n)|<k$. This
principle and RT$^n_k$ form the two ends of a spectrum of principles
RT$^n_{k,j}$ for $1 \leq j < k$, where an instance is a coloring $c$
of $[\mathbb N]^n$ with $k$ many colors, and a solution to this
instance is an infinite set $T$ such that $|c([T]^n)| \leq j$. It
would be interesting to pursue versions of HT based on these
principles. One might hope to show, for instance, that there is a
boundary between principles that ``behave like HT'', e.g.\ HT$^{\leq
  2}_4$, which as mentioned in the introduction was shown to imply
ACA$_0$ in \cite{CKLZ}; and those that ``behave like versions of TS /
RT'', e.g.\ the thin version of HT$^{\leq 2}_4$, which can easily be
shown to follow from RT$^2_{4,2}$.

\end{document}